\documentclass{article}
\usepackage[left=1in,right=1in,bottom=1in,top=1in]{geometry}

\usepackage[utf8]{inputenc}
\usepackage{amsmath,amsfonts,amsthm, amssymb}
\usepackage{hyperref}
\usepackage[capitalise]{cleveref}
\usepackage[dvipsnames]{xcolor}

\usepackage{biblatex}
\newtheorem{thm}{Theorem}
\newtheorem{lem}{Lemma}

\theoremstyle{definition}
\newtheorem{eg}{Example}

\usepackage{lipsum}
\usepackage{todonotes}

\title{Asymptotics of bivariate algebraico-logarithmic generating functions}

\author{Torin Greenwood\thanks{Department of Mathematics, North Dakota State University, Fargo, ND USA, \href{mailto:torin.greenwood@ndsu.edu}{torin.greenwood@ndsu.edu}}, Tristan Larson\thanks{Department of Mathematics, North Dakota State University, Fargo, ND USA, \href{mailto:tristan.larson@ndsu.edu}{tristan.larson@ndsu.edu}}}

\begin{document}

\maketitle

\begin{abstract}
We derive asymptotic formulae for the coefficients of bivariate generating functions with algebraic and logarithmic factors.  Logarithms appear when encoding cycles of combinatorial objects, and also implicitly when objects can be broken into indecomposable parts.  
Asymptotics are quickly computable and can verify combinatorial properties of sequences and assist in randomly generating objects.  While multiple approaches for algebraic asymptotics have recently emerged, we find that the contour manipulation approach can be extended to these D-finite generating functions.
\end{abstract}

A goal in analytic combinatorics in several variables (ACSV) is to derive asymptotic estimates for multivariate arrays that encode combinatorial information.  Asymptotic formulae are useful for computing highly accurate estimates of sequences, determining what large structures look like, and randomly generating objects.  In contrast, even when exact formulae can be found, they may be be cumbersome to evaluate or interpret
\cite{Wilf:1982}.  The schema for generating asymptotic estimates follows.
\begin{center}
    \includegraphics[width=\textwidth]{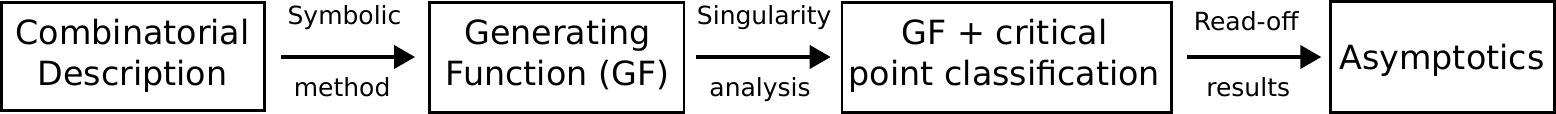}
\end{center}
The read-off results depend on the form of the generating function (GF).  Here, we broaden the read-off results to bivariate algebraico-logarithmic GFs, for several reasons:
\begin{enumerate}
    \item Algebraico-logarithmic GFs appear widely, including problems involving cycles, P\'{o}lya enumeration, P\'{o}lya urns \cite{Hwang:2022}, necklaces \cite{Hackl:2018}, and more.  Logarithms also appear in problems where a combinatorial family of objects can be described implicitly as components within other objects \cite[Section III.7.3]{FlSe:2009}.
    \item Our main theorem, \cref{thm:logH}, involves D-finite GFs, while most ACSV results have focused on rational or algebraic GFs.
    \item Differing approaches for computing algebraic asymptotics have recently emerged \cite{Greenwood:2018, GMRW:2022, BJP:2023}.  Although the direct contour manipulations of \cite{Greenwood:2018} are technical, here we show that this approach is readily applied to algebraico-logarithmic GFs.
\end{enumerate}
\section{Univariate asymptotics}

The \emph{symbolic method} is now standard \cite{FlSe:2009}, and it converts common combinatorial operations on sets into algebraic operations on GFs.  For example, if the GF $F(z)$ encodes the sequence counting objects of size $n$ from some set $\mathcal{F}$ as $n$ varies, then the GF $1/(1 - F(z))$ encodes the numbers of ordered sequences of objects of any length from $\mathcal{F}$ whose total size is $n$.  Other common operations include taking powersets, multisets, or ordered tuples from a set, or highlighting a component within a combinatorial object.  Of particular importance here is the \emph{cycle construction} for ordinary or exponential GFs, which involves sums of logarithms or a single logarithm, respectively.  Once we have a GF, we aim for asymptotics of the form
\begin{equation} \label{AsymptForm}
[z^n] F(z) \sim C n^r (\log^s n) \rho^n 
\end{equation}
as $n \to \infty$ where $C, r, s,$ and $\rho$ do not depend on $n$.  The location of a GF's closest singularities to the origin determines $\rho$, and the behavior of the GF near these singularities determines $C, r,$ and $s$.  \cref{eg:Catalan} below illustrates the efficiency of computing univariate asymptotics.\\
\hrule
\begin{eg}[Logarithms of Catalan numbers] \label{eg:Catalan} Logarithms of the Catalan number GF were considered in Knuth's 2014 Christmas Tree lecture \cite{Knuth:2014} and the American Mathematical Monthly \cite{Knuth:2015}.  They have since been studied \cite{Chu:2019} and can encode cycles of Dyck paths and families of labelled paths \cite{JaKo:2023}.  In particular, consider
\[
D^{(m)}(z) = \left[\log \left(\frac{1 - \sqrt{1 - 4z}}{2z} \right)\right]^m.
\]
The singularity at $z = 0$ is removable.  Otherwise, $D^{(m)}(z)$ has algebraic singularities determined by the zero set within the square root, $\{z: 1 - 4z = 0\}$, and when the input to the log is $0$, so $\{z: (1 - \sqrt{1 - 4z})/(2z) = 0\}$. The input to the logarithm is never zero (since the point $z = 0$ is a removable singularity), so the only singularity of $D^{(m)}(z)$ is at $z = 1/4$.  Expanding near $z = 1/4$ reveals
\begin{align*}
    D^{(m)}(z) &= \left(\log 2 + \sqrt{1 - 4z} - \frac{7}{2}(1-4z) + O(1 - 4z)^{3/2}\right)^m\\
        &= \log^m 2 - \binom{m}{1} [\log^{m-1} 2] \sqrt{1 - 4z} + \mbox{higher order terms in $(1 - 4z)$}.
\end{align*}
In this expansion, the next term with an algebraic singularity is $O(1 - 4z)^{3/2}$, so the transfer theorem from Flajolet and Odlyzko \cite{FlOd:1990} immediately yields
\[
[z^n] D(z) = \frac{m\log^{m-1} 2}{2\sqrt{\pi}} \cdot n^{-3/2} \cdot 4^n + O(4^n n^{-5/2}).
\]
\end{eg}
\hrule

\section{Multivariate asymptotics background} \label{sec:MultBack}

Multivariate GFs encode arrays of numbers, useful for tracking combinatorial parameters.  Alternatively, multivariate GFs can assist univariate analyses, including lattice walk enumeration \cite{BMMi:2010}.
Let $F(x, y)$ be the bivariate GF encoding the array $a_{r, s}$, so that $F(x, y) = \sum a_{r, s} x^r y^s.$  For a fixed \emph{direction} $\mathbf{\hat r} := (r_1, r_2) \in \mathbb{R}_{> 0}^2$, we search for an asymptotic expression for $[x^{r_1n}y^{r_2n}]F(x, y)$ as $n \to \infty$.  As in the univariate case, asymptotics are often in the form $C \cdot n^r (\log^s n) \cdot \rho^n$ for appropriate choices of constants $r, s,$ and $\rho$.  However, the idea of \emph{closest} singularity to the origin needs to be refined in multiple variables.

ACSV connects geometry and combinatorics, as there is a diverse set of possible geometries of GF singularities.  Our focus here is the simplest but most common scenario, when a GF has a smooth minimal critical point.  Other cases for rational GFs are covered in \cite{MePeWi:2024}.  Let $\mathcal{V} := \{(x, y) : H(x, y) = 0\}$ be a singular variety for a GF, defined by when some analytic function $H(x, y)$ is zero.  For rational GFs, $H$ is the denominator, while for algebraico-logarithmic GFs, $H$ may be the input of a square root, a logarithm, or the product of several such inputs.  Smooth critical points for the direction $\mathbf{\hat r} = (r_1, r_2)$ satisfy
\[
H = 0, \ \ r_2 xH_x = r_1 y H_y,
\]
where $H_x$ and $H_y$ represent the partial derivatives of $H$ with respect to $x$ and $y$ \cite{MePeWi:2024}.  By design, these equations yield points minimizing the exponential growth rate in \cref{eq:Cauchy} below. Additionally, a smooth critical point must be a location where $\mathcal{V}$ is locally a smooth manifold.  This can be checked by verifying that $H, H_x$ and $H_y$ never simultaneously vanish, since the implicit function theorem guarantees a local smooth parameterization near any point where at least one partial is nonzero.  When $H$ is a polynomial, all conditions for smooth critical points are also defined by polynomial equations, which means that identifying critical points can be done efficiently with Gr\"{o}bner bases.

\cref{thm:logH} requires that the critical point $(p, q)$ is \emph{minimal}, meaning there are no singularities coordinate-wise smaller in $\mathcal{V}$.  Minimality is simpler to check when a GF has only non-negative coefficients (the \emph{combinatorial} case).  Minimality is crucial here to allow an explicit Cauchy integral contour manipulation that reaches the critical point $(p, q)$.

\section{Generating function classifications and logarithms}

Multivariate rational GFs cover many combinatorial scenarios: not only do they count arrays enumerating the output of discrete finite automata, but they are also useful when more complicated sequences can be expressed as the diagonal of a rational GF.  Nonetheless, there are combinatorial situations where a sequence cannot be encoded as a diagonal of the terms of a rational GF, such as when the asymptotics of a sequence are not of the form $C n^{-s}\rho^{-n}$ for $s \in \mathbb{Z}/2$.  However, the dictionary of asymptotic results is incomplete for GFs beyond the rational domain.

A GF $F(\mathbf{z})$ can be classified according to what kind of equation $F$ satisfies.  Rational GFs satisfy linear equations with coefficients in $\mathbf{z}$, while algebraic GFs satisfy polynomial equations.  Even more broadly, D-finite GFs satisfy linear PDEs with coefficients in $\mathbf{z}$.

Several distinct approaches recently advanced asymptotic formulae for algebraic GFs. In \cite{Greenwood:2018}, a change of variables and a direct contour manipulation of the Cauchy integral formula lead to results for bivariate algebraic GFs.  This technical process is currently limited to two dimensions, but could be extended to more dimensions with some additional overhead.  Another possible approach \cite{GMRW:2022} embedded the coefficients of an algebraic GF into a rational GF in more variables.  Accessing the singular variety of an algebraic GF directly from its minimal polynomial sometimes gives faster and cleaner results \cite{BJP:2023}.  Finally, \cite{BeRi:1983, GaRo:1992, Dr:1994} give probabilistic interpretations of the coefficients of algebraic GFs.

It is less clear how to approach D-finite GFs.  Here, we modify the contour approach from \cite{Greenwood:2018} to attack a concrete class of bivariate D-finite GFs that include logarithms.  In contrast, it is not obvious how the other approaches could be adapted to this setting.
\section{Result}

Our focus is bivariate GFs of the form $F(x, y) = H(x, y)^{-\alpha}[\log H(x, y)]^\beta$, where $H(x, y)$ is analytic near the origin with only non-negative power series coefficients, and where $\beta \in \mathbb{Z}_{\geq 0}$ and $\alpha \in \mathbb{R}$ is not in $\mathbb{Z}_{\leq 0}$.  This form is motivated by the results in \cite{FlSe:2009}. 

\begin{thm} \label{thm:logH}
Let $H(x, y)$ be an analytic function near the origin whose power series expansion at $(0, 0)$ has non-negative coefficients.  Define $\mathcal{V} = \{(x, y) : H(x, y) = 0\}$.  Assume that there is a single smooth strictly minimal critical point of $\mathcal{V}$ at $(p, q)$ within the domain of analyticity of $H$ where $p$ and $q$ are real and positive.  Let $\lambda = \frac{r + O(1)}{s}$ as $r,s \to \infty$ with $r$ and $s$ integers. Define the following quantities:
    \begin{align*}
        \chi_1  &= \frac{H_y(p,q)}{H_x(p,q)} = \frac{p}{\lambda q}, \\
        \chi_2  &= \left.\frac{1}{2H_x}\left(\chi_1^2H_{xx} - 2\chi_1H_{xy} + H_{yy}\right)\right|_{(x,y)=(p,q)}, \\
        M       &= -\frac{2\chi_2}{p} - \frac{\chi_1^2}{p^2} - \frac{1}{\lambda q^2}.
    \end{align*}
Assume that $H_x(p,q)$ and $M$ are nonzero. Fix $\alpha \in \mathbb{R}$ where $\alpha \not \in \mathbb{Z}_{\leq 0}$ and $\beta \in \mathbb{Z}_{\geq 0}$.  Then, the following expression holds as $r,s \to \infty$:
\begin{align*}
    [x^ry^s]H(x,y)^{-\alpha}\log^\beta(H(x,y))
        \sim (-1)^\beta\frac{(-pH_x(p,q))^{-\alpha}r^{\alpha-1}}{\Gamma(\alpha)\sqrt{-2\pi q^2Mr}}p^{-r}q^{-s}\log^\beta r\left[1+\sum_{j\geq1}\frac{\mathcal{E}_j}{\log^jr}\right],
\end{align*}
where $\mathcal{E}_j = \sum_{k=0}^j\binom{\beta}{j}\binom{j}{k}\log^k\left(\frac{-1}{pH_x(p,q)}\right)\Gamma(\alpha)\left.\frac{d^{j-k}}{dt^{j-k}}\frac{1}{\Gamma(t)}\right|_{t=\alpha}$.
\end{thm}

The theorem applies when the singularity nearest the origin is determined by a zero in the logarithm.  In other cases, the singularity may be algebraic, in which case existing algebraic results apply (see \cref{eg:Narayana}).  Also, when the closest singularity is instead due to a pole of $H$, rewriting $\log(H) = -\log(1/H)$ allows us to apply \cref{thm:logH}.  

\cref{thm:logH} could be extended to the case where $\beta \not \in \mathbb{Z}_{\geq 0}$, although this adds additional complexity because the logarithmic term in the GF then contributes an additional branch cut.  When $\beta \in \mathbb{Z}_{\geq 0}$, the series in \cref{thm:logH} is finite, but in general it is infinite.  Also, the theorem statement is still true when $\alpha = 0$ (or even $\alpha \in \mathbb{Z}$) by defining $(1/\Gamma(\alpha))$ and $\left. d^j/dt^j (1/\Gamma(t)) \right|_{t = \alpha}$
by their limits at $\alpha = 0$, as described in the univariate case in \cite{FlSe:2009}.  The analogous asymptotic expansion for $\alpha = 0$ can then be computed, still with descending powers of $\log(r)$ and leading term given by
    \[
    [x^ry^s]\log^\beta H(x,y) \sim (-1)^\beta\frac{\beta r^{-3/2}p^{-r}q^{-s}}{\sqrt{-2\pi q^2M}}\log^{\beta-1}r.
    \]

\section{Examples}
The examples below and more are analyzed in the {\tt SageMath} worksheet here:
\begin{center}
\url{https://cocalc.com/Tristan-Larson/FPSAC-algebraico-logarithmic/Examples}
\end{center}

\begin{eg}[Necklaces] \label{eg:Necklace}
As in \cite{Hackl:2018}, consider necklaces with black and white beads where no two white beads are adjacent.  These are analyzed in \cite{Hackl:2018} via a univariate GF, and they can be constructed by a ``necklace process'' shown in \cref{fig:Necklace} that relates to network communication models.
\begin{figure}
\begin{center}
    \includegraphics[width=.8\textwidth]{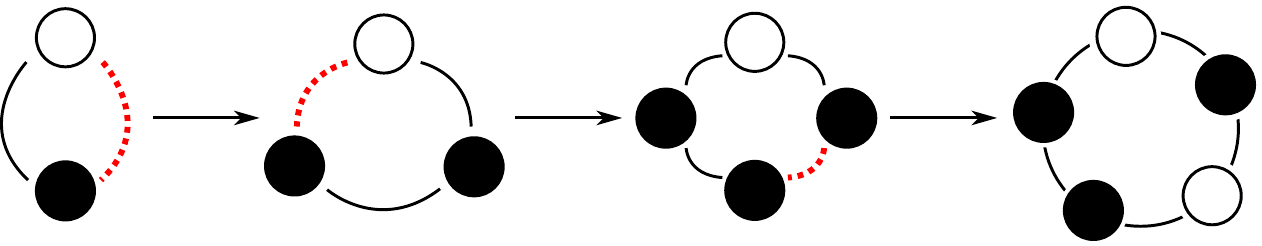}
\end{center}
\caption{Start with one white bead and one black bead.  Then, pick any edge in the necklace.  If both neighboring beads are black, insert a white bead on this edge.  Otherwise, insert a black bead.  The necklace illustrated above has multiple constructions but contributes once to the $x^5y^2$ term in the bivariate GF in \cref{eg:Necklace}.\\
}
\label{fig:Necklace}
\hrule
\end{figure}
Let $\varphi$ be Euler's totient function.  We consider the bivariate GF
    \[N(x,y) = \sum_{k\geq 1}\frac{\varphi(k)}{k}\log     \left(\frac{1-x^k}{1-x^k-y^kx^{2k}}\right),\]
in which the coefficient $[x^ry^s]N(x,y)$ counts the number of necklaces with $r$ total beads and $s$ white beads.  The GF can be derived by viewing necklaces as cycles of $\{$white beads followed by a positive number of black beads$\}$.

Consider seeking asymptotics in the direction $(\ell, 1)$ with $\ell > 2$.  Combinatorially, $\ell>2$ corresponds to necklaces having more than twice as many total beads as white beads.  For each $k$, define $H_k = 1-x^k-y^kx^{2k}$, which contributes $k^2$ critical points satisfying
    \[
    x^k = \frac{\ell-2}{\ell-1},\ y^k = \frac{\ell-1}{(\ell-2)^2}.
    \] 
  Define $(p_k, q_k)$ to be the positive real solution.  We verify that there are no nonsmooth critical points by checking for each $k$ that $[H_k = 0, \frac{\partial}{\partial x} H_k = 0, \frac{\partial}{\partial y} H_k = 0]$ has no solutions.
Furthermore, the exponential growth rate near any critical point $(p_k, q_k)$ is given by $|p_k^\ell q_k|^{-1}$, which can easily be verified as maximized when $k = 1$.
Thus, we need only to consider the contributions from the critical point $(p_1, q_1)$ determined by the first term in the sum, $\log\left(\frac{1-x}{1-x-yx^2}\right)$. For this, we find that
        \[
        \chi_1 = \frac{(\ell-2)^3}{\ell(\ell-1)^2},\ 
        \chi_2 = -\frac{(2\ell-1)(\ell-2)^5}{\ell^3(\ell-1)^3},\ \text{and}\ 
        M      = -\frac{(\ell-2)^5}{\ell^3(\ell-1)}.
        \]
Then the asymptotic enumeration formula of \cref{thm:logH} gives
    \[[x^{\ell n}y^n]N(x,y) \sim \frac{n^{-3/2}\ell^{5/2}}{\sqrt{2\pi}}\frac{(\ell-1)^{(2\ell n-2n+3)/2}}{(\ell-2)^{(2\ell n-4n+9)/2}}.\]

To illustrate accuracy, consider $\ell = 3$. The approximation becomes $[x^{pn}y^n]N(x,y) \sim 18n^{-3/2}4^n\sqrt{3/\pi},$
yielding (for instance) the estimate $[x^{225}y^{75}]N(x,y) \approx 6.199 \times 10^{41}.$
The actual value is $[x^{225}y^{75}]N(x,y) = 6.188\ldots \times 10^{41}$, with an error of only $0.167\%$.
\end{eg}

\begin{figure}
\begin{center}
\includegraphics[width=.95\textwidth]{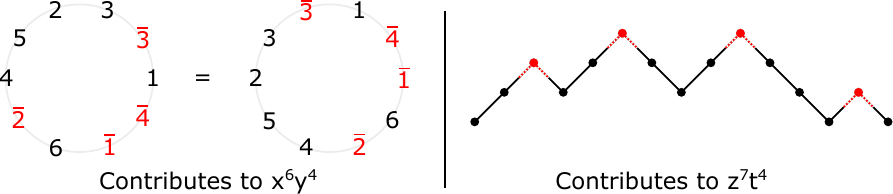}
\end{center}
\caption{On the left, a cyclical interlaced permutation is drawn as enumerated in \cref{eg:CyclicInterlaced}.  Because there are 6 black numbers and 4 red numbers, this contributes to $x^6y^4$.  Rotating the numbers gives the same configuration.  On the right, a Dyck path with 14 steps and 4 peaks is drawn, which contributes to $z^7t^4$ in the GF from \cref{eg:Narayana}.\\}\label{fig:NarayanaNumberRing}
\hrule
\end{figure}

\begin{eg}[Cyclical Interlaced Permutations] \label{eg:CyclicInterlaced}
    Let $\mathcal{C}$ be the set of circular arrangements of the bicolored set 
    $\{1, 2, \ldots, n, {\color{red}\overline{1}}, {\color{red}\overline{2}}, \ldots, {\color{red}\overline{m}}\}$,
    with $m + n \geq 1$, as illustrated in \cref{fig:NarayanaNumberRing}.  
    For fixed $m$ and $n$, there are $(n + m)!/(n + m)$ arrangements. 
    So, if $x$ tracks the number of black elements and $y$ tracks the number of the red (barred) elements, $\mathcal{C}$ has the GF,
        \[C(x,y) = \log\left(\frac{1}{1-x-y}\right) = \sum_{n+m\geq1}\frac{1}{n+m}\binom{n+m}{n}x^ny^m.\]
    Note that this GF is the logarithm of the GF in \cite[Examples 2.2, 8.13, 9.10]{MePeWi:2024}.  The labelled objects here lead to an exponential GF with a single logarithm, in contrast to the unlabelled objects in \cref{eg:Necklace} that yield an ordinary GF with a sum of logarithms.
    
    We will now compute the asymptotics in the direction $(1,\ell)$, where $\ell > 0$. There is a unique minimal smooth critical point at $(1/(1 + \ell), \ell/(1 + \ell))$%
    . For the quantities defined in \cref{thm:logH}, we have $\chi_1 = 1, \chi_2 = 0,$ and $M = -(1+\ell)^3/\ell,$ yielding
        \[[x^ry^{\ell r}]C(x,y) \sim \frac{r^{-3/2}\ell^{-r\ell}(1+\ell)^{(1+\ell)r}}{\sqrt{2\pi \ell(1+\ell)}}.\]
\end{eg}

\begin{eg}[Logarithm of Narayana numbers]  \label{eg:Narayana} The Narayana numbers refine \cref{eg:Catalan}: let $a_{n, s}$ be the number of Dyck paths with length $n$ and number of peaks $s$, as in \cref{fig:NarayanaNumberRing}.  This example illustrates how algebraic singularities may still determine asymptotics for a non-algebraic GF.  Let $N(z, t) := \sum a_{n, s} z^n t^s$ be the GF for the Narayana numbers, and let $P(z, t)$ count the Dyck paths that never return to the $x$-axis except at their start and end.  Then, from the symbolic method, $N$ and $P$ satisfy the relations,
\begin{align*}
N(z, t) &= \frac{1}{1 - P(z, t)}, \ \ P(z, t) = tz + z(N(z, t)-1).\\
N(z, t) &= \frac{1 + z - tz - \sqrt{(1 + z - tz)^2 - 4z}}{2z}.
\end{align*}
Consider the growth rate of the coefficients $[z^n t^s] \log^r N(z, t)$ in the direction $(\ell, 1)$ for any $\ell > 1$.
To determine the singularities of $\log^r N(z, t)$, note that $N$ has a removable singularity at $z = 0$, with limit $1$.  Thus, $\log^r N(z, t)$ has singularities from the logarithm determined by $N(z, t) = 0$ (with $z \neq 0)$ and algebraic singularities determined by the zero set of $H(z, t) := (1 + z - tz)^2 - 4z$.  A simple analysis determines that $N(z, t) \neq 0$ for any values of $z$ and $t$.  
Thus, we find a single smooth critical point at $(p, q) := ([1 - 1/\ell]^2, 1/[\ell-1]^2)$, and there are no nonsmooth critical points.

To use the results in \cite{Greenwood:2018}, we must also ensure that the critical point is minimal.  This is slightly more difficult here, but because $N(z, t)$ is combinatorial, there must be a minimal singularity with positive real coordinates by Pringsheim's Theorem.  To verify minimality, consider points of the form $(z, t) = (vp, wq)$ for real paramters $0 \leq v, w \leq 1$, and search for values of $v$ and $w$ where $H(vp, wq) = 0$.  Because $H$ is quadratic, we can solve for $v$ in terms of $w$ and $\ell$, and then verify that for all $\ell > 1$ and all $v, w \in [0, 1]$, $dw/dv < 0$.  Ultimately, this implies that there are no solutions where $v$ and $w$ are both less than $1$, so that $(p, q)$ indeed must be minimal.

Near the critical point $(p, q)$ we can expand $\log^r N(z, t)$ as in \cref{eg:Catalan} to obtain
\[
\log^r N(z, t) = \log^r \left(\frac{1 + p - pq}{2p}\right) - r \log^{r - 1} \left(\frac{1 + p - pq}{2p}\right) \frac{\sqrt{H}}{1 + p - pq} + \cdots,
\]
from which we conclude that the $\sqrt{H}$ term determines the dominant asymptotics.  Applying Corollary 2 of \cite{Greenwood:2018} yields our final result,
\[
[z^{\ell n} t^n] \log^r N(z, t) \sim \frac{r}{{2}\pi}\log^{r-1} \left(\frac{\ell}{\ell - 1}\right) \cdot n^{-2} \cdot \left(\ell - 1\right)^{-2n(\ell - 1){-1}} \ell^{2\ell n - {1}}.
\]

\end{eg}

\section{Proof sketch}

We prove \cref{thm:logH} by using the Cauchy integral formula,
\begin{equation} \label{eq:Cauchy}
[x^r y^s] H(x,y)^{-\alpha}\log^\beta(H(x,y)) = \frac{-1}{4\pi^2} \iint_T H(x,y)^{-\alpha}\log^\beta(H(x,y)) x^{-r - 1} y^{-s - 1} dx dy,
\end{equation}
where $T$ is a torus centered at $(0, 0)$ that is small enough that it does not enclose any singularities of $H(x,y)^{-\alpha}\log^\beta(H(x,y))$.

\subsection{Step 1: Change of variables}  A key idea in \cite{Greenwood:2018} is to use this change of variables, with $\chi_1$ and $\chi_2$ as in \cref{thm:logH}:
\[
u = x + \chi_1(y - q) + \chi_2(y-q)^2,\ \ \ \ 
v = y.
\]
Call $H(x, y) := \tilde H(u, v)$. Expanding $\tilde H(u, v) = \sum_{m, n \geq 0} d_{m, n} (u-p)^m (v-q)^n$, we find $d_{0, 0} = d_{0, 1} = d_{0, 2} = 0$.  For functions of the form $F(u, v) = [\log \tilde H(u, v)]^\beta (\tilde H(u, v))^{-\alpha}$, it turns out that having these three terms equal to zero is enough to approximate $F$ near $(p, q)$ with the product of a function in $u$ and a function in $v$.

\subsection{Step 2: Choose a convenient contour} \label{sec:Contour}
In order to justify that $F$ can be written as a product, we first decide how to deform the torus $T$ in \cref{eq:Cauchy}.  We focus on the details of the contour when $v$ is near the critical point $q$, since the contour away from the critical point does not contribute to the asymptotics.  We choose approximately a product contour, with a Hankel contour in the $u$ variable contour and a circle of radius $q$ in the $v$ variable.

The $u$ variable contour will wrap around a point that shifts slightly depending on the $v$ variable: more precisely, since $(p, q)$ is a smooth critical point, the zero set $\mathcal{V} := \{(u, v) : H(u, v) = 0\}$ can be parameterized with a smooth function $G(v)$ such that $H(p + G(v), v) = 0$ locally near $v = q$.  Thus, we center the $u$ contour at the point $p + G(v)$.  See \cref{fig:ContourPicture} for a diagram of the $u$ contour near the point $(p, q)$.  Because we assume $(p,q)$ is a unique minimal critical point of $H$, for $v$ values away from $q$, we can expand the contour to circles with radii larger than $|p|$ and $|q|$ making this portion of the contour negligible.  The transition between regimes is described in greater detail in \cite{Greenwood:2018}.

\begin{figure}
\begin{center}
    \includegraphics[width=0.65\textwidth]{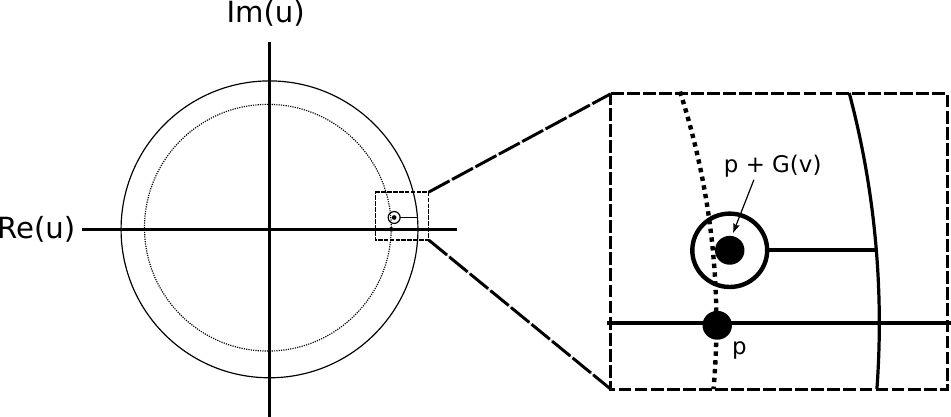}
\end{center}
\caption{The torus $T$ in the Cauchy integral is deformed so that near the critical point, it is approximately a product contour.  For $v$ close to $q$, the contour is exactly a circle.  For $u$ close to $p$, the contour expands beyond the critical point $p$ by using a Hankel-like wrapping around the zero set of $H$, which is parameterized in terms of $v$ by $p + G(v)$.\\}
\label{fig:ContourPicture}
\hrule
\end{figure}

\subsection{Step 3: Approximate the integrand with a product integral} \label{sec:ProductIntegrand}

After the change of variables to $(u, v)$ coordinates, we can estimate the resulting Cauchy integrand with a product of a function in $u$ and a function in $v$.
\begin{lem} \label{lem:Product}
Let $\mathcal{C}_r$ be the portion of the contour defined in \cref{sec:Contour} where $v$ is close to $q$.  Then,
\begin{multline*}
\iint_{T} ( \tilde H(u, v)^{-\alpha} \log^\beta (\tilde H(u, v))) (u - \chi_1(v-q) - \chi_2(v-2)^2)^{-r-1} v^{-s-1} dudv\\
\sim \iint_{\mathcal{C}_r} \bigg([H_x(p, q) (u-p)]^{-\alpha} \log^\beta (H_x(p, q) (u-p)) u^{-r-1} v^{-s-1}  \times\\
\left[1 - \frac{\chi_1(v-q) + \chi_2(v-q)^2}{p}\right]^{-r-1}\bigg) dudv.
\end{multline*}
\end{lem}

The full proof of this lemma is technical, generalizing a similar proof in \cite{Greenwood:2018}.  
Near $(p, q)$, tedious computations reveal $\tilde H(u, v)$ may be estimated by truncating its power series.  
Away from $(p, q)$, the contributions to the integral are exponentially smaller than the parts near $(p, q)$, and hence they may be ignored.
The addition of the logarithm adds some new technicalities to the proof.  To begin, we define correction factors $K, L,$ and $N$:
\[
K = \left(\frac{1 - \frac{\chi_1(v-q) + \chi_2(v-q)^2}{u}}{1 - \frac{\chi_1(v-q) + \chi_2(v-q)^2}{p}} \right)^{-r-1},\ \  L = \left(\frac{\tilde H(u, v)}{C(u-p)}\right)^{-\alpha},\ \ 
N = \left(\frac{\log \tilde H(u, v)}{\log [C(u-p)]} \right)^\beta,
\]
with $C := H_x(p, q)$.  The point of these factors is that the integrands of the left and right sides in \cref{lem:Product} are equal up to $K \cdot L \cdot N$.  Thus, in a neighborhood near $(p, q)$, the goal is to show that $K, L, N = 1 + o(1)$ uniformly.  Lemmas 4 and 5 of \cite{Greenwood:2018} state the result for $K$ and $L$, so it remains to show the equivalent result for $N$.
\begin{lem} \label{lem:CorrectionN}
When $u$ and $v$ are sufficiently close to $(p, q)$, the following holds uniformly as $r, s \to \infty$ with $\lambda = \frac{r + O(1)}{s}$:
\[
N(u, v) = 1 + o(1).
\]
\end{lem}
\begin{proof}
    Again with $C := H_x(p, q)$, we write
    \begin{align*}
        \log \tilde H(u, v) &= \log [C(u-p)] + \log \frac{\tilde H(u, v)}{C(u-p)}
        = \log [C(u-p)] - \frac{1}{\alpha} \log L(u, v),
    \end{align*}
    \[
    N(u, v) = \left[ 1 - \frac{\frac{1}{\alpha}\log L(u, v)}{\log C(u-p)}\right]^\beta.
    \]
    From Lemma 5 of \cite{Greenwood:2018}, $L(u, v) = 1 + o(1)$ in this region.  Thus,
    $\log L(u, v) = o(1)$
    as $r \to \infty$.  Additionally, $|\log C(u - p)|$ is bounded away from zero as $u$ is close to $p$, implying that $N(u, v) = [1 + o(1)]^\beta = 1 + o(1)$ as desired.
\end{proof}

\subsection{Step 4: Evaluate the product integral}
We can now split \cref{eq:Cauchy} into two univariate integrals.  The $v$ integral is a standard Fourier-Laplace type integral that is identical to the case without logarithms.
\begin{lem}[Lemma 9 of \cite{Greenwood:2018}] \label{lem:vInt}
The following holds uniformly as $r, s \to \infty$ with $\lambda = \frac{r + O(1)}{s}$:
\[
\int_V v^{-s-1} \left[1 - \frac{\chi_1(v-q) + \chi_2(v-q)^2}{p} \right]^{-r-1} dv = iq^{-s} \sqrt{\frac{2\pi}{-q^2Mr}} + o\left(q^{-s}r^{-\frac{1}{2}}\right).
\]

\end{lem}

Thus, the final step in proving \cref{thm:logH} is to evaluate the $u$ integral.
\begin{lem} \label{lem:uInt} Define $\mathcal{E}_j = \sum_{k=0}^j\binom{\beta}{j}\binom{j}{k}\log\left(\frac{-1}{pH_x(p,q)}\right)\Gamma(\alpha)\left.\frac{d^{j-k}}{dt^{j-k}}\frac{1}{\Gamma(t)}\right|_{t=\alpha}$.  Then, as $r \to \infty$,
\begin{multline*}
    \frac{1}{2\pi i}\int_U\left(H_x(p,q)(u-p)\right)^{-\alpha}\left[\log^\beta(H_x(p,q)(u-p))\right]u^{-r-1}du \\
    \sim (-1)^\beta\frac{(-pH_x(p,q))^{-\alpha}r^{\alpha-1}}{\Gamma(\alpha)}p^{-r}\left[\log^\beta r\right]\left[1+\sum_{j\geq1}\frac{\mathcal{E}_j}{\log^jr}\right],
\end{multline*}
where $U$ is a small circle near $0$ that does not enclose any other singularities of the integrand.
\end{lem}

\begin{proof} We begin with some factoring:
    \begin{multline*}
\frac{1}{2\pi i}\int_U\left(H_x(p,q)(u-p)
        \right)^{-\alpha}\left[\log^\beta(H_x(p,q)(u-p))\right]u^{-r-1}du \\
    = (-1)^\beta\frac{\left(-pH_x(p,q)\right)^{-\alpha}}{2\pi i}\int_U(1-u/p)^{-\alpha}\left[\log\frac{1}{1-u/p}+L\right]^\beta u^{-r-1}du, \\
\end{multline*}
 where $L=\log\frac{-1}{pH_x(p,q)}$. Substitute $u=pz$ and expand $\left[\log\frac{1}{1-z}+L\right]^\beta$ as a series:
    \[(-1)^\beta\frac{(-pH_x(p,q))^{-\alpha}}{2\pi i}p^{-r}\sum_{k\geq0}\binom{\beta}{k}L^k\int_U(1-z)^{-\alpha}\log^{\beta-k}\frac{1}{1-z} z^{-r-1}dz.\]

By \cite[Theorem 3A]{FlOd:1990}, we have for each $k$ that
    \[\frac{1}{2\pi i}\int_U(1-z)^{-\alpha}\log^{\beta-k}\frac{1}{1-z} z^{-r-1}dz \sim \frac{r^{\alpha-1}}{\Gamma(\alpha)}\log^{\beta-k}r\left[1+\sum_{j\geq1}\frac{c^{(k)}_j}{\log^j r}\right],\]
where $c^{(k)}_j := \binom{\beta-k}{j}\Gamma(\alpha)\left.\frac{d^j}{dt^j}\frac{1}{\Gamma(t)}\right|_{t=\alpha}$. So
\begin{align*}
    &(-1)^\beta\frac{(-pH_x(p,q))^{-\alpha}}{2\pi i}p^{-r}\sum_{k\geq0}\binom{\beta}{k}L^k\int_U(1-z)^{-\alpha}\log^{\beta-k}\frac{1}{1-z} z^{-r-1}dz \\
    &\quad\sim (-1)^\beta\frac{(-pH_x(p,q))^{-\alpha}r^{\alpha-1}}{\Gamma(\alpha)}p^{-r}\sum_{k\geq0}\binom{\beta}{k}L^k\log^{\beta-k}r\left[1+\sum_{j\geq1}\frac{c^{(k)}_j}{\log^j r}\right].
\end{align*}

Letting $e^{(k)}_j = \binom{\beta}{k}L^kc^{(k)}_j$, we can rewrite the double sum as
\begin{align*}
     \sum_{k\geq0}\binom{\beta}{k}L^k\log^{\beta-k}r\left[1+\sum_{j\geq1}\frac{c^{(k)}_j}{\log^j r}\right] &= \log^\beta r\sum_{k\geq0}\binom{\beta}{k}L^k\left[\frac{1}{\log^kr}+\sum_{j\geq k+1}\frac{c^{(k)}_{j-k}}{\log^j r}\right] \\ 
    = \log^\beta r\left[1+\sum_{j\geq1}\sum_{k=0}^j\frac{e^{(k)}_{j-k}}{\log^jr}\right] &= \log^\beta r\left[1+\sum_{j\geq1}\frac{\mathcal{E}_j}{\log^jr}\right],
\end{align*}
with $\mathcal{E}_j$ defined above. With this, we have the result as desired.
\end{proof}

\printbibliography

@Misc{BJP:2023,
    AUTHOR = {Yuliy Baryshnikov and Kaitian Jin and Robin Pemantle},
     TITLE = {Coefficient asymptotics of algebraic multivariable generating functions},
howpublished = {Preprint (online)},
       URL = {https://ymb.web.illinois.edu/wp-content/uploads/2023/02/AlgebraicGF.pdf},
}

@article{BeRi:1983,
title = {Central and local limit theorems applied to asymptotic enumeration II: Multivariate generating functions},
journal = {Journal of Combinatorial Theory, Series A},
volume = {34},
number = {3},
pages = {255-265},
year = {1983},
issn = {0097-3165},
doi = {https://doi.org/10.1016/0097-3165(83)90062-6},
url = {https://www.sciencedirect.com/science/article/pii/0097316583900626},
author = {Edward A Bender and L {Bruce Richmond}}
}

@article{BMMi:2010,
  	title={Walks with small steps in the quarter plane},
  	author={Bousquet-M{\'e}lou, M. and Mishna, M.},
  	journal={Contemporary Mathematics},
  	volume={520},
  	pages={1--40},
  	year={2010}}

@article {Chu:2019,
    AUTHOR = {Chu, Wenchang},
     TITLE = {Logarithms of a binomial series: extension of a series of
              {K}nuth},
   JOURNAL = {Math. Commun.},
  FJOURNAL = {Mathematical Communications},
    VOLUME = {24},
      YEAR = {2019},
    NUMBER = {1},
     PAGES = {83--90},
      ISSN = {1331-0623,1848-8013},
   MRCLASS = {05A15 (05A10)},
  MRNUMBER = {3884556},
MRREVIEWER = {Tri\ Lai},
}

@article{Dr:1994,
title = {Asymptotic distributions and a multivariate darboux method in enumeration problems},
journal = {Journal of Combinatorial Theory, Series A},
volume = {67},
number = {2},
pages = {169-184},
year = {1994},
issn = {0097-3165},
doi = {https://doi.org/10.1016/0097-3165(94)90011-6},
url = {https://www.sciencedirect.com/science/article/pii/0097316594900116},
author = {Michael Drmota}
}

@article{FlOd:1990,
	Author = {Flajolet, Philippe and Odlyzko, Andrew},
	Coden = {SJDMEC},
	Fjournal = {SIAM Journal on Discrete Mathematics},
	Issn = {0895-4801},
	Journal = {SIAM J. Discrete Math.},
	Mrclass = {05A15 (30E20 40E05 41A60)},
	Mrnumber = {MR1039294 (90m:05012)},
	Mrreviewer = {E. Rodney Canfield},
	Number = {2},
	Pages = {216--240},
	Title = {Singularity analysis of generating functions},
	Volume = {3},
	Year = {1990},
    URL = {https://doi.org/10.1137/0403019}}

@book{FlSe:2009,
	Author = {Flajolet, Phillipe and Sedgewick, Robert},
	Isbn = {0521898064},
	Pages = {824},
	Publisher = {Cambridge University Press},
	Title = {Analytic combinatorics},
	Year = {2009}}

@article{Greenwood:2018,
title = {Asymptotics of bivariate analytic functions with algebraic singularities},
journal = {Journal of Combinatorial Theory, Series A},
volume = {153},
pages = {1-30},
year = {2018},
issn = {0097-3165},
doi = {https://doi.org/10.1016/j.jcta.2017.06.014},
url = {https://www.sciencedirect.com/science/article/pii/S0097316517300833},
author = {Torin Greenwood}
}

@article{GMRW:2022,
	Author = {Torin Greenwood and Tiadora Ruza and Stephen Melczer and Mark C. Wilson},
	Journal = {S\'{e}minaire Lotharingien de Combinatoire (Proceedings of FPSAC 2022)},
	Pages = {12},
	Title = {Asymptotics of Coefficients of Algebraic Series Via Embedding Into Rational Series (Extended Abstract)},
	Volume = {86B},
	Year = {2022},
    URL = {https://www.mat.univie.ac.at/~slc/wpapers/FPSAC2022/30.pdf}}

@article {Hackl:2018,
    AUTHOR = {Hackl, Benjamin and Prodinger, Helmut},
     TITLE = {The necklace process: a generating function approach},
   JOURNAL = {Statist. Probab. Lett.},
  FJOURNAL = {Statistics \& Probability Letters},
    VOLUME = {142},
      YEAR = {2018},
     PAGES = {57--61},
      ISSN = {0167-7152},
   MRCLASS = {60C05 (05A15 05A16)},
  MRNUMBER = {3842616},
MRREVIEWER = {Adam Hammett},
       DOI = {10.1016/j.spl.2018.06.010},
       URL = {https://doi.org/10.1016/j.spl.2018.06.010},
}

@article{GaRo:1992,
title = {Central and local limit theorems applied to asymptotic enumeration IV: multivariate generating functions},
journal = {Journal of Computational and Applied Mathematics},
volume = {41},
number = {1},
pages = {177-186},
year = {1992},
issn = {0377-0427},
doi = {https://doi.org/10.1016/0377-0427(92)90247-U},
url = {https://www.sciencedirect.com/science/article/pii/037704279290247U},
author = {Zhicheng Gao and L.Bruce Richmond}
}

@misc{Hwang:2022,
      title={Analysis of some exactly solvable diminishing urn models}, 
      author={Hsien-Kuei Hwang and Markus Kuba and Alois Panholzer},
      year={2022},
      eprint={2212.05091},
      archivePrefix={arXiv},
      primaryClass={math.CO}
}

@misc{JaKo:2023,
      title={Logarithms of Catalan generating functions: A combinatorial approach}, 
      author={Sabine Jansen and Leonid Kolesnikov},
      year={2023},
      eprint={2302.09661},
      archivePrefix={arXiv},
      primaryClass={math.CO}
}

@misc{Knuth:2014,
  author       = {Knuth, D.E.},
  title        = {3/2-ary trees},
  howpublished = {Annual Christmas Tree lecture},
  month        = {12},
  year         = {2014},
  url = {https://www.youtube.com/watch?v=P4AaGQIo0HY}
}

@article {Knuth:2015,
    AUTHOR = {Knuth, D.E.},
     TITLE = {Log–squared of the {C}atalan generating function},
   JOURNAL = {Amer. Math. Monthly},
    VOLUME = {122},
      YEAR = {2015},
     PAGES = {390}
}

@book{MePeWi:2024, place={Cambridge}, edition={2}, series={Cambridge Studies in Advanced Mathematics}, title={Analytic Combinatorics in Several Variables}, publisher={Cambridge University Press}, author={Pemantle, Robin and Wilson, Mark C. and Melczer, Stephen}, year={2024}, collection={Cambridge Studies in Advanced Mathematics}}

@article {Wilf:1982,
    AUTHOR = {Wilf, Herbert S.},
     TITLE = {What is an answer?},
   JOURNAL = {Amer. Math. Monthly},
  FJOURNAL = {American Mathematical Monthly},
    VOLUME = {89},
      YEAR = {1982},
    NUMBER = {5},
     PAGES = {289--292},
      ISSN = {0002-9890,1930-0972},
   MRCLASS = {05A15 (68C25)},
  MRNUMBER = {653502},
       DOI = {10.2307/2321713},
       URL = {https://doi.org/10.2307/2321713},
}

\end{document}